\documentclass[english]{article}
\usepackage[T1]{fontenc}
\usepackage[latin9]{inputenc}
\usepackage{amsmath}
\usepackage{amsthm}
\usepackage{esint}
\usepackage{natbib}
\usepackage{anysize}
\makeatletter
\theoremstyle{plain}
\newtheorem{thm}{\protect\theoremname}
\theoremstyle{remark}

\theoremstyle{definition}
\newtheorem{defn}[thm]{\protect\definitionname}
\theoremstyle{plain}
\newtheorem{conjecture}[thm]{\protect\conjecturename}
\ifx\proof\undefined
\newenvironment{proof}[1][\protect\proofname]{\par
\normalfont\topsep6\p@\@plus6\p@\relax
\trivlist
\itemindent\parindent
\item[\hskip\labelsep\scshape #1]\ignorespaces
}{%
\endtrivlist\@endpefalse
}
\providecommand{\proofname}{Proof}
\fi

\makeatother
 
\newtheorem*{remark}{Remark}

\usepackage{babel}
\providecommand{\conjecturename}{Conjecture}
\providecommand{\definitionname}{Definition}
\providecommand{\theoremname}{Theorem}

\setcitestyle{authoryear,open={(},close={)}}

\title{Discrete time approximation of a COGARCH(p,q) model and its estimation (Preliminary Version)} 

\author{Stefano M. Iacus\footnote{Department of Economics, Management and Quantitative Methods.
University of Milan. CREST Japan Science and Technology Agency. E-mail: stefano.iacus@unimi.it.},  
 Lorenzo Mercuri\footnote{Department of Economics, Management and Quantitative Methods. University of Milan.
University of Milan. E-mail: lorenzo.mercuri@unimi.it.}  and Edit Rroji\footnote{Department of Economics, Business, Mathematical and Statistical Sciences. University of Trieste. E-mail: erroji@units.it.}}

\begin{document}
\maketitle 

\begin{abstract} 
In this paper, we construct a sequence of discrete time stochastic processes that converges in probability and in the Skorokhod metric to a COGARCH(p,q) model. The result is  useful for the estimation of the continuous model defined for irregularly spaced time series data. The estimation procedure is based on the maximization of a pseudo log-likelihood function and is implemented in the \texttt{yuima} package.
\end{abstract} 

\section{Introduction}
%
The COGARCH(1,1) model has been introduced by \cite{Cogarch2004} as a continuous time counterpart of the GARCH(1,1) process. The continuous time model preserves the main features of the GARCH model since the same underlying noise drives the variance and the return processes. For the COGARCH(1,1) case, different methods for its estimation have been proposed. For instance, \cite{Haug2007} develop a procedure based on the matching of theoretical and empirical moments. \cite{maller2008garch} use an approximation scheme for obtaining estimates of parameters through the maximization of a pseudo-loglikelihood function while \cite{muller2010mcmc} develop a Markov Chain Monte Carlo estimation
procedure based on the same approximation scheme.\newline 
The  COGARCH(1,1) model has been generalized to the higher order case by \cite{Chadraa2010Thesis} and \cite{brockwell2006}. Based on our knowledge, this is the only estimation method for higher order models and it is based on the matching of empirical and theoretical moments.\newline
In this paper, we construct a sequence of discrete time stochastic processes that converges in probability and in the Skorokhod metric to a COGARCH(p,q) model. Our results generalize the approach in  \cite{maller2008garch} for building a sequence of discrete time stochastic processes based on a GARCH(1,1) model that converges in the Skorokhod metric to its continuous counterpart, i.e COGARCH(1,1) model.\newline
Results derived for a COGARCH(p,q) model in \cite{Chadraa2010Thesis} are used in this paper for extending the estimation procedure based on the maximization of the pseudo log-likelihood function. This estimation method is then implemented in the \texttt{yuima} package available on CRAN \citep[See][ for more details on \texttt{yuima} package]{Brousteetal2013,BrousteIacus2013,IacusMercur2015,iacus2015estimation}. \newline
The outline of the paper is as follows. In Section 
\ref{sec:1} we review some useful properties needed in in Section \ref{sec:2} where we introduce a discrete version of our process and prove the convergence to the COGARCH(p,q) model using the Skorokhod metric.

\section{Preliminaries}
\label{sec:1}

In this section we review useful results for obtaining a sequence
of discrete time processes that converges in Skorokhod distance \citep[see for example]{Billingsley68book} to a COGARCH(p,q) model.
\begin{defn}
The sequence of random vectors $Q_{n}$ is uniformly convergent in
probability to $Q$ if and only if: 
\begin{equation}
\underset{\theta\in\Theta}{\sup}\left\Vert Q_{n,\theta}-Q_{\theta}\right\Vert \stackrel{P}{\rightarrow}0,\label{eq:ucp}
\end{equation}
where $\left\Vert \cdot \right\Vert $ is the Euclidean norm. 
\end{defn}
\begin{conjecture}
The definition holds also for any vector norm $\left\Vert \cdot \right\Vert _{A}$
induced by an invertible matrix $A$, i.e. $\left\Vert x\right\Vert _{A}$=$\left\Vert Ax\right\Vert $ where
$A$ is a non singular matrix.\end{conjecture}
\begin{defn}
Let $\mathcal{\left\Vert  \cdot \right\Vert }$ be
a norm on $\mathcal{R}^{n}$, we introduce induced norm $\left\Vert \cdot \right\Vert_M $
as a function from $\mathcal{R}^{n\times n}$ to $\mathcal{R}_{+}$ defined
as:
\[
\left\Vert A\right\Vert_M :=\underset{\left\Vert x\right\Vert \neq0}{\sup}\frac{\left\Vert Ax\right\Vert}{\left\Vert x\right\Vert}=\underset{\left\Vert z\right\Vert =1}{\sup}\left\Vert Az\right\Vert
\]
where $A\in\mathcal{R}^{q\times q}$.\end{defn}
\begin{thm}
The induced norm $\left\Vert \cdot \right\Vert_M $ satisfies the following
properties \citep[see ][]{BookDesoer1975}:
\newline
1) $\left\Vert Ax\right\Vert \leq\left\Vert A\right\Vert_M \left\Vert x\right\Vert$
\newline
2) $\left\Vert \alpha A\right\Vert_M \leq\left|\alpha\right|\left\Vert A\right\Vert $
\newline
3) $\left\Vert A+B\right\Vert_M \leq\left\Vert A\right\Vert_M +\left\Vert B\right\Vert_M $
\newline
4) $\left\Vert AB\right\Vert_M \leq\left\Vert A\right\Vert_M \left\Vert B\right\Vert_M $, \newline
where $A\in\mathcal{R}^{q\times q}$, $B\in\mathcal{R}^{q\times q}$ and $\alpha$ is a scalar.
\end{thm}
We have also that any induced vector norm satisfies the following inequality:
\citep[see][]{Denis2002}:
\begin{equation}
\left\Vert \frac{e^{At}-I}{t}-A\right\Vert_M \leq\frac{e^{\left\Vert At\right\Vert_M }-1-\left\Vert At\right\Vert_M }{\left|t\right|},\ t\in\mathcal{R}.\label{eq:InqInducedVectorNorm}
\end{equation}
\begin{defn}
Let $\left\Vert \cdot \right\Vert_M $ be the induced vector norm by the norm
$\left\Vert \cdot \right\Vert$ defined on $\mathcal{R}^{n}$, the logarithmic norm $\mu\left(A\right)$ \citep[see][ for its properties]{Strom1975LogarithmicNorm} is defined as:
\[
\mu\left(A\right):=\underset{t\rightarrow 0^{+}}{\lim}\frac{\left\Vert I+At\right\Vert_M -1}{t}.
\]
\end{defn}
\begin{thm}
For the logarithmic norm the following inequalities hold:
\[
\left\Vert e^{At}\right\Vert_M \leq e^{\mu\left(A\right)t}\leq e^{\left\Vert A\right\Vert_M t}
\]
\end{thm}
Let $a_{n}$and $b_{n}$ be sequences of non-negative numbers for
$n=1,\ldots,N$. Define as a linear recursive equation the sequence
$y_{n}$:
\[
y_{n}=a_{n}y_{n-1}+b_{n}
\]
with initial condition $y_{0}=c$ where $c$ is a scalar.
\begin{thm}
If we have that $a_{n}\geq1$ and $b_{n}\geq0$, the sequence $y_{n}$ is
non decreasing with 
\[
y_{N}=\underset{n=1,\ldots,N}{\max}y_{n}
\]
and 
\[
y_{N}=\left[\prod_{k=0}^{N-1}a_{N-k}\right]y_{0}+b_{N}+\sum_{j=1}^{N-1}\left[\prod_{h=1}^{j}a_{N+1-h}\right]b_{N-j}.
\]
\end{thm}
%
%
%
%
%
%
%

\section{Main Result}
\label{sec:2}

We recall the definition of a COGARCH(p,q) process, introduced in \cite{brockwell2006}, based on the following equations:
\begin{align}
\mbox{d}G_{t} & =\sqrt{V_{t}}\mbox{d}L_{t}\nonumber \\
V_{t} & =\alpha_{0}+\mathbf{a^{\top}}Y_{t-}\nonumber \\
\mbox{d}Y_{t} & =BY_{t-}\mbox{d}t+\mbox{e}\left(\alpha_{0}+\mathbf{a^{\top}}Y_{t-}\right)\mbox{d}\left[L,L\right]^{d}\label{eq:COGARCH_p_q}
\end{align}
where $B\in\mathcal{R}^{q\times q}$ is matrix of the form:
\[
B=\left[\begin{array}{ccccc}
0 & 1 & 0 & \ldots & 0\\
0 & 0 & 1 & \ldots & 0\\
\vdots & \vdots & \vdots & \ddots & \vdots\\
0 & 0 & 0 & \ldots & 1\\
-b_{q} & -b_{q-1} & \ldots & \ldots & -b_{1}
\end{array}\right]
\]
and $\mathbf{a}$ and $\mathbf{\mathbf{e}}$ are vectors  defined as:
\begin{align*}
\mathbf{a} & =\left[a_{1},\ldots,a_{p},a_{p+1},\ldots,a_{q}\right]^{\top}\\
\mathbf{e} & =\left[0,\ldots,0,1\right]^{\top}
\end{align*}for $a_{p+1}=\ldots=a_{q}=0$.
As remarked in \cite{brockwell2006} the state process
$Y_{t}$ in a COGARCH(p,q) model is:
\[
Y_{t}=J_{s,t}Y_{s}+K_{s,t}\ s\leq t
\]
 where $J_{s,t}\in\mathcal{R}^{q\times q}$ is a random matrix and $K_{s,t}\in\mathcal{R}^{q\times1}$ is a random vector.
\newline
In particular, if the driven noise is a Compound Poisson the matrices and vectors
in the state process have an analytical form. Let $N$ be the number
of jumps of a Compound Poisson in the interval $\left[0,t\right]$. 
Define $\tau_{N}$ as the time of the last jump in this interval interval
and $Z_{N}:=\Delta L_{\tau_{N}}^{2}=\left(L_{\tau_{N}}-L_{\tau_{N}-}\right)^{2}$
the square of the jump at time $\tau_{N}$. The process $Y_{t}$ can
be rewritten as follows:
\[
Y_{t}=e^{B\left(t-\tau_{N}\right)}Y_{\tau_{N}}\ t\in\left[\tau_{N},\tau_{N+1}\right)
\]
where $Y_{\tau_{N}}$ is the state process at jump time $\tau_{N}$, i.e. the last jump of size less or equal to $t$, defined as:
\begin{equation}
Y_{\tau_{N}}=C_{N}Y_{\tau_{N-1}}+D_{N}\label{eq:StateProcYatTau}
\end{equation}
where the random coefficients $C_{N}$ and $D_{N}$ in \eqref{eq:StateProcYatTau}
are respectively:
\begin{align}
C_{N} & =\left(I+Z_{N}\mathbf{ea}^{\top}\right)e^{B\Delta\tau_{N}}\nonumber \\
D_{N} & =\alpha_{0}Z_{N}\mathbf{e}.\label{eq:TRUE_D_C}
\end{align}
As in \cite{maller2008garch}, we construct a sequence of discrete
processes that converges to the COGARCH(p,q) model in \eqref{eq:COGARCH_p_q}
by means of the Skorokhod distance \citep[see ][ for more details]{Billingsley68book}. For each $n\geq0$
we consider a sequence of natural numbers $N_{n}$ such that $\underset{n\rightarrow+\infty}{\lim}N_{n}=+\infty$
and we obtain a partition of the interval $\left[0,T\right]$ defined
as:
\begin{equation}
0=t_{0,n}\leq t_{1,n}\leq\ldots\leq t_{N_{n},n}=T.\label{eq:Partition}
\end{equation}
The mesh of this partition is:
\[
\Delta t_{n}:=\underset{i=1,\ldots,N_{n}}{\max}\Delta t_{i,n}\ \underset{n\rightarrow+\infty}{\rightarrow}0.
\]
Using the partition in \eqref{eq:Partition}, we introduce the process
$G_{i,n}$ as follows:
\begin{equation}
G_{i,n}=G_{i-1,n}+\sqrt{V_{i-1.n}\Delta t_{i,n}}\epsilon_{i,n}\label{eq:G_i_n}
\end{equation}
where the innovations $\epsilon_{i,n}$ are constructed using the first
jump approximation method developed in \cite{szimayer2007finite} that we review here quickly.
\newline
Let $m_{n}$ be a strict positive sequence of real numbers satisfying the conditions:
\begin{align*}
 & m_{n}\leq1 \ \forall n \geq 0,\\
 & \underset{n\rightarrow+\infty}{\lim}m_{n}=0.
\end{align*}
We require the L\'evy measure $\Pi$ to satisfy following property:
\[
\underset{n\rightarrow+\infty}{\lim}\Delta t_{n}\bar{\Pi}^{2}\left(m_{n}\right)=0
\]
where $\bar{\Pi}\left(x\right):=\int_{\left|y\right|>x}\Pi\left(\mbox{d}x\right)$. 
\newline
We define the stopping time process: 
\begin{equation}
\tau_{i,n}:=\inf\left\{ t\in\left[t_{i-1,n},t_{i,n}\right):\left|\Delta L_{t}\right|>m_{n}\right\} \label{eq:StoppingTimeTaui_n}
\end{equation}
and construct a sequence of independent random variables $\left(\mathbf{1}_{\tau_{i,n}<+\infty}\Delta L_{\tau_{i,n}}\right)_{i=1,\ldots,N_{n}}$ with
density:
\[
f\left(x\right)=\frac{\Pi\left(\mbox{d}x\right)}{\bar{\Pi}\left(m_{n}\right)}\left(1-e^{\Delta t_{i,n}\bar{\Pi}\left(m_{n}\right)}\right).
\]
We introduce the innovations $\epsilon_{i,n}$ defined as:
\begin{equation}
\epsilon_{i,n}=\frac{\mathbf{1}_{\tau_{i,n}<+\infty}\Delta L_{\tau_{i,n}}-v_{i,n}}{\eta_{i,n}}\label{eq:InnovEpsilon_i_n}
\end{equation}
where $v_{i,n}$ and $\eta_{i,n}$ are respectively the mean and the variance of $\epsilon_{i,n}$.
 The variance process $V_{t}$ in \eqref{eq:COGARCH_p_q} is approximated by the process $V_{i,n}$ as:
\begin{equation}
V_{i,n}=\alpha_{0}+\mathbf{a}^{\top}Y_{i,n}\label{eq:V_i_n}
\end{equation}where $Y_{i,n}$ is given by:
\begin{equation}
Y_{i,n}=C_{i,n}Y_{i-1,n}+D_{i,n},\label{eq:IMPORTANTEYI_N}
\end{equation}with coefficients:
\begin{align}
C_{i,n} & =\left(I+\epsilon_{i,n}^{2}\Delta t_{i,n}\mathbf{ea}^{\top}\right)e^{B\Delta t_{i,n}}\nonumber \\
D_{i,n} & =\alpha_{0}\epsilon_{i,n}^{2}\Delta t_{i,n}\mathbf{e}.\label{eq:C_i_nAndD_i_n}
\end{align}
The couple $\left(G_{i,n},V_{i,n}\right)$ converges to the couple $\left(G_{t},V_{t}\right)$
in the Skorokhod distance. The Skorokhod distance between
two processes $U,V$ defined on $D^{d}\left[0,T\right]$, i.e. space
of c\`adl\`ag $\mathcal{R}^{d}$ stochastic processes on $\left[0,T\right]$,
is 
\[
\rho\left(U,V\right):=\underset{\lambda\in\Lambda}{\inf}\left\{ \underset{0\leq t\leq T}{\sup}\left\Vert U_{t}-V_{\lambda\left(t\right)}\right\Vert +\underset{0\leq t\leq T}{\sup}\left|\lambda\left(t\right)-t\right|\right\} 
\]
where $\Lambda$ is a set of increasing continuous functions with $\lambda\left(0\right)=0$
and $\lambda\left(T\right)=T$.\newline 
First of all we need the following auxiliar result.

\begin{thm}\label{thm:AuxResult1} 
Let $N_{n}\left(t\right)$ be a counting process defined as:
\[
N_{n}\left(t\right):=\#\left\{ i\in\mathcal{N}:\ \tau_{i,n}^{\star}\leq t\right\} 
\]
where $t\leq T$, $N_{n}\left(0\right)=0$, $N_{n}\left(T\right)=N_{n}$
and $\tau_{i,n}^{\star}=\min\left\{ \tau_{i,n},t_{i,n}\right\} $
with $\tau_{i,n}$ and $t_{i,n}$ in \eqref{eq:StoppingTimeTaui_n} and \eqref{eq:Partition} respectively. \newline Let $L_{t}$ be a Compound Poisson with finite second moment, the
positive process $H_{n}\left(t\right)$ defined as:
\[
H_{n}\left(t\right):=\prod_{k=1}^{N_{n}\left(t\right)}C_{k,n}^{\star}
\]
where 
\[
C_{k,n}^{\star}:=\left(1+\epsilon_{k,n}^{2}\Delta t_{k,n}\left\Vert \mathbf{ea^{\top}}\right\Vert _{M}\right)e^{\left\Vert B\right\Vert _{M}\Delta t_{i,n}}
\]
 converges uniformly in probability on a compact interval $\left[0,T\right]$
(hereafter ucp) to the positive process $\tilde{H}_{n}\left(t\right)$
as:
\[
\tilde{H}_{n}\left(t\right):=\prod_{k=1}^{N_{n}\left(t\right)}\tilde{C}_{k,n}
\]
with
\[
\tilde{C}_{k,n}:=\left(1+\mathbf{1}_{\tau_{k,n}<+\infty}\Delta L_{\tau_{k,n}}^{2}\left\Vert \mathbf{ea^{\top}}\right\Vert _{M}\right)e^{\left\Vert B\right\Vert _{M}\Delta t_{i,n}}
\]
i.e.
\[
\underset{t\in\left[0,T\right]}{\sup}\left|H_{n}\left(t\right)-\tilde{H}_{n}\left(t\right)\right|\stackrel{p}{\rightarrow}0.
\]
For each fixed $n$, $\tilde{H}_{n}\left(t\right)$ is a non decreasing
striclty positive process in the compact interval $\left[0,T\right]$
such that $\forall t\in\left[0,T\right]$:
\[
\tilde{H}_{n}\left(t\right)\leq\tilde{H}_{n}\left(T\right)\leq e^{\left\Vert B\right\Vert _{M}T+\sum_{0\leq s\leq T}\ln\left(1+\Delta L_{s}^{2}\left\Vert \mathbf{ea^{\top}}\right\Vert _{M}\right)}
\]
 \end{thm}
\begin{proof}
We start from
\begin{align*}
\underset{t\in\left[0,T\right]}{\sup}\left|H_{n}\left(t\right)-\tilde{H}_{n}\left(t\right)\right| & =\underset{t\in\left[0,T\right]}{\sup}\left|\prod_{k=1}^{N_{n}\left(t\right)}C_{k,n}^{\star}-\prod_{k=1}^{N_{n}\left(t\right)}\tilde{C}_{k,n}\right|\\
 & \leq e^{\left\Vert B\right\Vert _{M}T}\underset{t\in\left[0,T\right]}{\sup}\left|\prod_{k=1}^{N_{n}\left(t\right)}\left(1+\epsilon_{k,n}^{2}\Delta t_{k,n}\left\Vert \mathbf{ea^{\top}}\right\Vert _{M}\right)-\prod_{k=1}^{N_{n}\left(t\right)}\left(1+\mathbf{1}_{\tau_{k,n}<+\infty}\Delta L_{\tau_{k,n}}^{2}\left\Vert \mathbf{ea^{\top}}\right\Vert _{M}\right)\right|\\
 & =e^{\left\Vert B\right\Vert _{M}T}\underset{t\in\left[0,T\right]}{\sup}\left|e^{\sum_{k=1}^{N_{n}\left(t\right)}\ln\left(1+\epsilon_{k,n}^{2}\Delta t_{k,n}\left\Vert \mathbf{ea^{\top}}\right\Vert _{M}\right)}-e^{\sum_{k=1}^{N_{n}\left(t\right)}\left(1+\mathbf{1}_{\tau_{k,n}<+\infty}\Delta L_{\tau_{k,n}}^{2}\left\Vert \mathbf{ea^{\top}}\right\Vert _{M}\right)}\right|.
\end{align*}

Observe that 
\begin{align*}
L_{n} & :=\underset{t\in\left[0,T\right]}{\sup}\left|\sum_{k=1}^{N_{n}\left(t\right)}\ln\left(1+\epsilon_{k,n}^{2}\Delta t_{k,n}\left\Vert \mathbf{ea^{\top}}\right\Vert _{M}\right)-\sum_{k=1}^{N_{n}\left(t\right)}\left(1+\mathbf{1}_{\tau_{k,n}<+\infty}\Delta L_{\tau_{k,n}}^{2}\left\Vert \mathbf{ea^{\top}}\right\Vert _{M}\right)\right|\\
 & \leq\underset{t\in\left[0,T\right]}{\sup}\left|\sum_{k=1}^{N_{n}\left(t\right)}\left(\epsilon_{k,n}^{2}\Delta t_{k,n}\left\Vert \mathbf{ea^{\top}}\right\Vert _{M}-\mathbf{1}_{\tau_{k,n}<+\infty}\Delta L_{\tau_{k,n}}^{2}\left\Vert \mathbf{ea^{\top}}\right\Vert _{M}\right)\right|\\
 & \leq\left\Vert \mathbf{ea^{\top}}\right\Vert _{M}\underset{t\in\left[0,T\right]}{\sup}\sum_{k=1}^{N_{n}\left(t\right)}\left|\left(\epsilon_{k,n}^{2}\Delta t_{k,n}-\mathbf{1}_{\tau_{k,n}<+\infty}\Delta L_{\tau_{k,n}}^{2}\right)\right|.
\end{align*}
As shown in \cite{maller2008garch}, we have that
\[
\underset{t\in\left[0,T\right]}{\sup}\sum_{k=1}^{N_{n}\left(t\right)}\left|\left(\epsilon_{k,n}^{2}\Delta t_{k,n}-\mathbf{1}_{\tau_{k,n}<+\infty}\Delta L_{\tau_{k,n}}^{2}\right)\right|\stackrel{p}{\rightarrow}0,
\]
that implies
\begin{equation}
\underset{t\in\left[0,T\right]}{\sup}\left|\sum_{k=1}^{N_{n}\left(t\right)}\ln\left(1+\epsilon_{k,n}^{2}\Delta t_{k,n}\left\Vert \mathbf{ea^{\top}}\right\Vert _{M}\right)-\sum_{k=1}^{N_{n}\left(t\right)}\left(1+\mathbf{1}_{\tau_{k,n}<+\infty}\Delta L_{\tau_{k,n}}^{2}\left\Vert \mathbf{ea^{\top}}\right\Vert _{M}\right)\right|\stackrel{p}{\rightarrow0}.
\label{llll}
\end{equation}
Using result in \eqref{llll}, we have
\[
\underset{t\in\left[0,T\right]}{\sup}\left|e^{\sum_{k=1}^{N_{n}\left(t\right)}\ln\left(1+\epsilon_{k,n}^{2}\Delta t_{k,n}\left\Vert \mathbf{ea^{\top}}\right\Vert _{M}\right)}-e^{\sum_{k=1}^{N_{n}\left(t\right)}\left(1+\mathbf{1}_{\tau_{k,n}<+\infty}\Delta L_{\tau_{k,n}}^{2}\left\Vert \mathbf{ea^{\top}}\right\Vert _{M}\right)}\right|\stackrel{p}{\rightarrow0}.
\]

$\tilde{H}_{n}\left(t\right)$ is a non decreasing strictly
positive process since is a product of terms $\tilde{C}_{k,n}\geq1$
a.s. and if $s>t$ then $\tilde{H}_{n}\left(s\right)$ has at least
the same terms as in $\tilde{H}_{n}\left(t\right)$
. Moreover 
\[
\tilde{H}_{n}\left(T\right)=e^{\left\Vert B\right\Vert _{M}T+\sum_{k=1}^{N_{n}}\ln\left(1+\mathbf{1}_{\tau_{k,n}<+\infty}\Delta L_{\tau_{k,n}}^{2}\left\Vert \mathbf{ea^{\top}}\right\Vert _{M}\right)}\leq e^{\left\Vert B\right\Vert _{M}T+\sum_{0\leq s\leq T}\ln\left(1+\Delta L_{s}^{2}\left\Vert \mathbf{ea^{\top}}\right\Vert _{M}\right)}
\]
since $\Delta L_{s}^{2}=\Delta L_{s}^{2}\mathbf{1}_{\left|\Delta L_{s}\right|\geq m_{n}}+\Delta L_{s}^{2}\mathbf{1}_{\left|\Delta L_{s}\right|<m_{n}}$.
\end{proof}
\begin{remark}
We observe, from Theorem \ref{thm:AuxResult1}, that 
\begin{equation}
H_{n}\left(t\right)\stackrel{ucp}{\rightarrow}\tilde{H}_{n}\left(t\right)\leq e^{\left\Vert B\right\Vert _{M}T+\sum_{0\leq s\leq T}\ln\left(1+\Delta L_{s}^{2}\left\Vert \mathbf{ea^{\top}}\right\Vert _{M}\right)}
\label{uffa}
\end{equation}
on an interval $\left[0,T\right]$. Moreover, the term on the right hand
side of the inequality in \eqref{uffa} is bounded almost surely on the compact interval $\left[0, T\right]$
since $L_{t}$ is a Compound Poisson process. \end{remark}
The following theorem is established for the Compound Poisson driven noise case.
\begin{thm}
\label{thm:Principal1} Let $L_{t}$ be a Compound Poisson process
with $E\left(L_{1}^{2}\right)<+\infty$. The Skorokhod distance computed on the
processes $\left(G_{t},V_{t}\right)_{t\geq0}$ and their discretized
version $\left(G_{i,n},V_{i,n}\right)_{i=1,\ldots,N_{n}}$ converges
in probability to zero, i.e.:
\[
\rho\left(\left(G_{i,n},V_{i,n}\right)_{i=1,\ldots,N_{n}},\left(G_{t},V_{t}\right)_{t\geq0}\right)\stackrel{P}{\rightarrow}0\ as\ n\rightarrow+\infty.
\]

\end{thm}
\begin{proof}
The proof follows the same steps as in \cite{maller2008garch}
\begin{itemize}
\item Approximation procedure for the underlying process.
\item Approximation procedure for the variance process.
\item Approximation procedure for the COGARCH(p,q) model.
\item Convergence of the pair in the Skorokhod distance.
\end{itemize}
Steps 1, 2, 4 are exactly the same  as in \cite{maller2008garch}. To prove that
the discrete variance process $V_{i,n}$ converges ucp
on a compact time interval to the continuous-time process $V_{t}$ we
 first need to show that $Y_{i,n}\stackrel{ucp}{\rightarrow}Y_{t}$. This
result is achieved through intermediate steps illustrated below. \newline We introduce the
counting process $N_{n}\left(t\right)$ defined as:
\begin{equation}
N_{n}\left(t\right):=\#\left\{ i\in\mathcal{N}:\tau_{i,n}^{\star}\leq t\right\} 
\label{eq:CountProcess}
\end{equation}
with $t\leq T$, $N_{n}\left(0\right)=0$ and $\tau_{i,n}^{\star}=\min\left\{ \tau_{i,n},t_{i,n}\right\} .$
\newline $N_{n}\left(t\right)$ increases by 1 in each subinterval $\left(t_{i-1,n},t_{i,n}\right]$, $i = 1, 2, \ldots, n$, at 
the first time the jump is of magnitude greater or equal to $m_{n}$ or
at $t_{i,n}$ if that jump does not occur. 
\newline Using the process $N_{n}\left(t\right)$ in \eqref{eq:CountProcess}
we construct the time process $\Gamma_{t,n}$ as:
\begin{equation}
\Gamma_{t,n}=\sum_{i=1}^{N_{n}\left(t\right)}\Delta t_{i,n}.\label{eq:Gamma_t_n}
\end{equation}
Now we want to show that the piecewise constant process $Y_{t,n}:=Y_{i,n}$ with
$t\in\left[t_{i,n},t_{i+1,n}\right)$ converges in ucp to the process
$\bar{Y}_{t,n}:=e^{B\left(t-\Gamma_{t,n}\right)}Y_{i,n}$ i.e.:
\begin{align*}
\underset{0\leq t\leq T}{\sup} & \left\Vert Y_{t,n}-\bar{Y}_{t,n}\right\Vert \stackrel{P}{\rightarrow}0.
\end{align*}
For each $t\in\left[0,T\right],$ we have:
\footnotesize
\begin{align*}
\left\Vert Y_{t,n}-\bar{Y}_{t,n}\right\Vert  & =\left\Vert e^{B\left(t-\Gamma_{t,n}\right)}Y_{i,n}-Y_{i,n}\right\Vert \\
 & \leq\left\Vert e^{B\left(t-\Gamma_{t,n}\right)}-I\right\Vert _{M}\left\Vert Y_{i,n}\right\Vert \\
 & =\left\Vert e^{B\left(t-\Gamma_{t,n}\right)}-I-B\left(t-\Gamma_{t,n}\right)+B\left(t-\Gamma_{t,n}\right)\right\Vert _{M}\left\Vert Y_{i,n}\right\Vert \\
 & \leq\left(\left\Vert e^{B\left(t-\Gamma_{t,n}\right)}-I-B\left(t-\Gamma_{t,n}\right)\right\Vert _{M}+\left\Vert B\left(t-\Gamma_{t,n}\right)\right\Vert \right)\left\Vert Y_{i,n}\right\Vert 
\end{align*}
\normalsize
using the inequality in \eqref{eq:InqInducedVectorNorm}, we get:
\begin{align}
\left\Vert Y_{t,n}-\bar{Y}_{t,n}\right\Vert  & \leq\left(e^{\left\Vert B\left(t-\Gamma_{t,n}\right)\right\Vert _{M}}-1\right)\left\Vert Y_{i,n}\right\Vert \nonumber \\
 & \leq\left(e^{\left\Vert B\right\Vert _{M}\Delta t_{n}}-1\right)\left\Vert Y_{i,n}\right\Vert \label{eq:IneqA}
\end{align}
Since by construction $Y_{t,n}=Y_{i,n}$ with $t \in \left[t_{i,n}, t_{i+1,n}\right)$ and $Y_{t,n}$ has c\`adl\`ag paths,
it follows that $\underset{t\in\left[0,T\right]}{\sup}\left\Vert Y_{t,n}\right\Vert $
is almost surely finite and
\begin{align*}
\underset{t\in\left[0,T\right]}{\sup}\left\Vert Y_{t,n}-\bar{Y}_{t,n}\right\Vert  & \leq\left(e^{\left\Vert B\right\Vert _{M}\Delta t_{n}}-1\right)\underset{t\in\left[0,T\right]}{\sup}\left\Vert Y_{t,n}\right\Vert \stackrel{P}{\rightarrow}0
\end{align*}
as $n\rightarrow+\infty$. 
\newline
The next step is to show the convergence ucp of $\bar{Y}_{t,n}$
to $\tilde{Y}_{t,n}$ where the last process is defined as:
\begin{equation}
\tilde{Y}_{t,n}=e^{B\left(t-\Gamma_{t,n}\right)}\tilde{Y}_{i,n}\label{eq:Y_tilde_t_n}
\end{equation}
with:
\begin{equation}
\tilde{Y}_{i,n}=\tilde{C}_{i,n}\tilde{Y}_{i-1,n}+\tilde{D}_{i,n}\label{eq:Y_tilde_i_n}
\end{equation}
where the random matrix $\tilde{C}_{i,n}$ and the random vector $\tilde{D}_{i,n}$
are respectively:
\begin{align}
\tilde{C}_{i,n} & =\left(I+\left(\mathbf{1}_{\tau_{i,n}<+\infty}\Delta L_{\tau_{i,n}}\right)^{2}\mathbf{ea}^{\top}\right)e^{B\Delta t_{i,n}}\nonumber \\
\tilde{D}_{i,n} & =\alpha_{0}\left(\mathbf{1}_{\tau_{i,n}<+\infty}\Delta L_{\tau_{i,n}}\right)^{2}\mathbf{e}.\label{eq:Ctilde_Dtilde}
\end{align}
We consider
\begin{equation}
\underset{t\in\left[0,T\right]}{\sup}\left\Vert \tilde{Y}_{t,n}-\bar{Y}_{t,n}\right\Vert \leq e^{\left\Vert B\right\Vert _{M}\Delta t_{n}}\underset{i=1,\ldots,N_{n}}{\sup}\left\Vert \tilde{Y}_{i,n}-Y_{i,n}\right\Vert \label{eq:SecondSup}
\end{equation}
and observe that, for $i=1,\ldots,N_{n}$, we have:
\begin{align}
\left\Vert \tilde{Y}_{i,n}-Y_{i,n}\right\Vert  & \leq\left\Vert \tilde{C}_{i,n}\tilde{Y}_{i-1,n}-C_{i,n}Y_{i-1,n}\right\Vert +\left\Vert \tilde{D}_{i,n}-D_{i,n}\right\Vert. \label{eq:absdiffY_i_n_tilde_Y_i_n}
\end{align}
We analyze the second term in \eqref{eq:absdiffY_i_n_tilde_Y_i_n} and get:
\footnotesize
\begin{align}
\left\Vert \tilde{D}_{i,n}-D_{i,n}\right\Vert  & =\left\Vert \alpha_{0}\left(\mathbf{1}_{\tau_{i,n}<+\infty}\Delta L_{\tau_{i,n}}\right)^{2}\mathbf{e}-\alpha_{0}\epsilon_{i,n}^{2}\Delta t_{i,n}\mathbf{e}\right\Vert \nonumber \\
 & \leq\left|\alpha_{0}\right|\left|\left(\mathbf{1}_{\tau_{i,n}<+\infty}\Delta L_{\tau_{i,n}}\right)^{2}-\epsilon_{i,n}^{2}\Delta t_{i,n}\right|.\label{eq:diffD_i_n_tildeD_i_n}
\end{align}
\normalsize
The first term in \eqref{eq:absdiffY_i_n_tilde_Y_i_n} can be bounded
by adding and subtracting the quantity $C_{i,n}\tilde{Y}_{i-1,n}$:
\footnotesize
\begin{align}
\left\Vert \tilde{C}_{i,n}\tilde{Y}_{i-1,n}-C_{i,n}Y_{i-1,n}\right\Vert  & =\left\Vert \tilde{C}_{i,n}\tilde{Y}_{i-1,n}-C_{i,n}\tilde{Y}_{i-1,n}+C_{i,n}\tilde{Y}_{i-1,n}-C_{i,n}Y_{i-1,n}\right\Vert \nonumber \\
 & \leq\left\Vert \tilde{C}_{i,n}-C_{i,n}\right\Vert _{M}\left\Vert \tilde{Y}_{i-1,n}\right\Vert +\left\Vert C_{i,n}\right\Vert _{M}\left\Vert \tilde{Y}_{i-1,n}-Y_{i-1,n}\right\Vert \nonumber \\
 & \leq\left\Vert \left[\left(\mathbf{1}_{\tau_{i,n}<+\infty}\Delta L_{\tau_{i,n}}\right)^{2}-\epsilon_{i,n}^{2}\Delta t_{i,n}\right]\mathbf{ea}^{\top}e^{B\Delta t_{i,n}}\right\Vert _{M}\left\Vert \tilde{Y}_{i-1,n}\right\Vert \nonumber \\
 & +\left\Vert C_{i,n}\right\Vert _{M}\left\Vert \tilde{Y}_{i-1,n}-Y_{i-1,n}\right\Vert \nonumber \\
 & \leq\left|\left(\mathbf{1}_{\tau_{i,n}<+\infty}\Delta L_{\tau_{i,n}}\right)^{2}-\epsilon_{i,n}^{2}\Delta t_{i,n}\right|\left\Vert \mathbf{ea}^{\top}\right\Vert _{M}e^{\left\Vert B\right\Vert _{M}\Delta t_{i,n}}\left\Vert \tilde{Y}_{i-1,n}\right\Vert \nonumber \\
 & +\left\Vert C_{i,n}\right\Vert _{M}\left\Vert \tilde{Y}_{i-1,n}-Y_{i-1,n}\right\Vert \label{eq:DiffCtildeC}
\end{align}
\normalsize
Substituting \eqref{eq:DiffCtildeC} and \eqref{eq:diffD_i_n_tildeD_i_n}
into \eqref{eq:absdiffY_i_n_tilde_Y_i_n} we have:
\footnotesize
\begin{align}
\left\Vert \tilde{Y}_{i,n}-Y_{i,n}\right\Vert  & \leq\left\Vert C_{i,n}\right\Vert _{M}\left\Vert \tilde{Y}_{i-1,n}-Y_{i-1,n}\right\Vert \nonumber \\
 & +\left|\left(\mathbf{1}_{\tau_{i,n}<+\infty}\Delta L_{\tau_{i,n}}\right)^{2}-\epsilon_{i,n}^{2}\Delta t_{i,n}\right|\left(\left|\alpha_{0}\right|+\left\Vert \mathbf{ea}^{\top}\right\Vert _{M}e^{\left\Vert B\right\Vert _{M}\Delta t_{i,n}}\left\Vert \tilde{Y}_{i-1,n}\right\Vert \right)\label{eq:DiffYtildeY2}
\end{align}
\normalsize
Since a.s.:
\begin{equation}
\left\Vert C_{i,n}\right\Vert _{M}\leq\left(1+\epsilon_{i,n}^{2}\Delta t_{i,n}\left\Vert \mathbf{ea}^{\top}\right\Vert _{M}\right)e^{\left\Vert B\right\Vert _{M}\Delta t_{i,n}}:=C_{i,n}^{\star}\geq1\label{eq:CoeffGreaterThan1}
\end{equation}
and defining 
\begin{equation}
K_{i-1,n}:=\left|\alpha_{0}\right|+\left\Vert \mathbf{ea}^{\top}\right\Vert _{M}e^{\left\Vert B\right\Vert _{M}\Delta t_{i,n}}\left\Vert \tilde{Y}_{i-1,n}\right\Vert \label{eq:K_i_n}
\end{equation}
 we have:
\begin{align}
\left\Vert \tilde{Y}_{i,n}-Y_{i,n}\right\Vert  & \leq C_{i,n}^{\star}\left\Vert \tilde{Y}_{i-1,n}-Y_{i-1,n}\right\Vert \nonumber \\
 & +\left|\left(\mathbf{1}_{\tau_{i,n}<+\infty}\Delta L_{\tau_{i,n}}\right)^{2}-\epsilon_{i,n}^{2}\Delta t_{i,n}\right|K_{i-1,n}.\label{eq:DiffYtildeY3}
\end{align}
The right hand side in \eqref{eq:DiffYtildeY3} is a linear recursive
equation with random coefficients and condition \eqref{eq:CoeffGreaterThan1}
implies that:
\footnotesize
\begin{align}
\underset{i=1,\ldots,N_{n}}{\sup}\left\Vert \tilde{Y}_{i,n}-Y_{i,n}\right\Vert  & \leq\left[\prod_{i=0}^{N_{n}-1}C_{N_{n}-i,n}^{\star}\right]\left\Vert \tilde{Y}_{0,n}-Y_{0,n}\right\Vert +\left|\left(\mathbf{1}_{\tau_{N_{n},n}<+\infty}\Delta L_{\tau_{N_{n},n}}\right)^{2}-\epsilon_{N_{n},n}^{2}\Delta t_{N_{n},n}\right|K_{N_{n}-1,n}\nonumber \\
 & +\sum_{i=1}^{N_{n}-1}\left[\prod_{h=1}^{i}C_{N_{n}+1-h,n}^{\star}\right]\left|\left(\mathbf{1}_{\tau_{N_{n}-i,n}<+\infty}\Delta L_{\tau_{N_{n}-i,n}}\right)^{2}-\epsilon_{N_{n}-i,n}^{2}\Delta t_{N_{n}-i,n}\right|K_{N_{n}-1-i,n}.\label{eq:CrucialPoint1}
\end{align}
\normalsize
The term:
\[
\left[\prod_{i=0}^{N_{n}-1}C_{N_{n}-i,n}^{\star}\right]\left\Vert \tilde{Y}_{0,n}-Y_{0,n}\right\Vert \geq0\ n\geq 1
\]
with
\[
E\left[\left(\prod_{i=0}^{N_{n}-1}C_{N_{n}-i,n}^{\star}\right)\left\Vert \tilde{Y}_{0,n}-Y_{0,n}\right\Vert \right]=E\left[\left(\prod_{i=0}^{N_{n}-1}C_{N_{n}-i,n}^{\star}\right)\right]\left\Vert \tilde{Y}_{0,n}-Y_{0,n}\right\Vert 
\]
since $\tilde{Y}_{0,n}=Y_{0,n}$ we have:
\begin{equation}
E\left[\left(\prod_{i=0}^{N_{n}-1}C_{N_{n}-i,n}^{\star}\right)\right]\left\Vert \tilde{Y}_{0,n}-Y_{0,n}\right\Vert =0\Rightarrow\left(\prod_{i=0}^{N_{n}-1}C_{N_{n}-i,n}^{\star}\right)\left\Vert \tilde{Y}_{0,n}-Y_{0,n}\right\Vert =0\ \text{a.s.}\label{eq:PROVA}
\end{equation}
Condition \eqref{eq:CrucialPoint1} becomes:
\footnotesize
\begin{align}
\underset{i=1,\ldots,N_{n}}{\sup}\left\Vert \tilde{Y}_{i,n}-Y_{i,n}\right\Vert  & \leq\left|\left(\mathbf{1}_{\tau_{N_{n},n}<+\infty}\Delta L_{\tau_{N_{n},n}}\right)^{2}-\epsilon_{N_{n},n}^{2}\Delta t_{N_{n},n}\right|K_{N_{n}-1,n}\nonumber \\
 & +\sum_{i=1}^{N_{n}-1}\left[\prod_{h=1}^{i}C_{N_{n}+1-h,n}^{\star}\right]\left|\left(\mathbf{1}_{\tau_{N_{n}-i,n}<+\infty}\Delta L_{\tau_{N_{n}-i,n}}\right)^{2}-\epsilon_{N_{n}-i,n}^{2}\Delta t_{N_{n}-i,n}\right|K_{N_{n}-1-i,n}.\label{eq:CrucialPoint2}
\end{align}
\normalsize
Defining: 
\footnotesize
\begin{align*}
Q_{n} & :=\left|\left(\mathbf{1}_{\tau_{N_{n},n}<+\infty}\Delta L_{\tau_{N_{n},n}}\right)^{2}-\epsilon_{N_{n},n}^{2}\Delta t_{N_{n},n}\right|K_{N_{n}-1,n}\\
 & +\sum_{i=1}^{N_{n}-1}\left[\prod_{h=1}^{i}C_{N_{n}+1-h,n}^{\star}\right]\left|\left(\mathbf{1}_{\tau_{N_{n}-i,n}<+\infty}\Delta L_{\tau_{N_{n}-i,n}}\right)^{2}-\epsilon_{N_{n}-i,n}^{2}\Delta t_{N_{n}-i,n}\right|K_{N_{n}-1-i,n}.
\end{align*}
\normalsize
we observe that $Q_{n}$ can be bounded. Indeed, $\forall i=1,\ldots,N_{n}$:
\[
\prod_{h=1}^{i}C_{N_{n}+1-h,n}^{\star}\leq\prod_{h=1}^{N_{n}}C_{N_{n}+1-h,n}^{\star}
\]
and, from Theorem \ref{thm:AuxResult1}, the quantity $\prod_{h=1}^{N_{n}}C_{N_{n}+1-h,n}^{\star}$ converges in probability to a non negative r.v. that is a.s. bounded by: 
\[
e^{\left\Vert B\right\Vert _{M}T+\sum_{0\leq s\leq T}\ln\left(1+\Delta L_{s}^{2}\left\Vert \mathbf{ea^{\top}}\right\Vert _{M}\right)}.
\]  
Even $\underset{i=1,\ldots,N_{n}}{\sup}K_{i,n}$ is bounded a.s. $\forall n$. Consequently we have:
\footnotesize
\begin{equation}
Q_{n} \leq \left[\prod_{h=1}^{N_{n}}C_{N_{n}+1-h,n}^{\star}\right]\left[\underset{i=1,\ldots,N_{n}}{\sup}K_{i,n}\right]\sum_{i=1}^{N_{n}}\left|\left(\mathbf{1}_{\tau_{i,n}<+\infty}\Delta L_{\tau_{i,n}}\right)^{2}-\epsilon_{i,n}^{2}\Delta t_{i,n}\right|\label{eq:Prova2}
\end{equation}
\normalsize
Since $\underset{n\rightarrow +\infty}{\lim} \ \underset{i=1,\ldots,N_{n}}{\sup}K_{i,n}=M<+\infty$
a.s. and, as shown in \cite{maller2008garch},
\[
\underset{t\in\left[0,T\right]}{\sup}{\sum_{i=1}^{N_{n}\left(t\right)}\left|\left(\mathbf{1}_{\tau_{i,n}<+\infty}\Delta L_{\tau_{i,n}}\right)^{2}-\epsilon_{i,n}^{2}\Delta t_{i,n}\right|}\stackrel{p}{\rightarrow}0
\]
as $n\rightarrow + \infty$, then $Q_{n}\stackrel{p}{\rightarrow}0$
that implies $\bar{Y}_{t,n}\stackrel{ucp}{\rightarrow}\tilde{Y}_{t,n}$.
\newline
We observe that, since the driven noise is a Compound Poisson, we
have only a finite number of jumps in a compact interval $\left[0,T\right].$
We indicate with $\tau_{k}$ the time of the $k$-th jump. Since the
irregular grid becomes finer as $n$ increases and satisfies
the following two conditions:
\begin{align*}
\Delta t_{n} & :=\underset{i=1,\ldots,N_{n}}{\max}\Delta t_{i,n}\underset{n\rightarrow+\infty}{\rightarrow}0\\
T & =\sum_{i=1}^{N_{n}}\Delta t_{i,n},
\end{align*}
then exists $n^{\star}$ such that for $n\geq n^{\star}$, all jump times $\tau_{k}\in\left\{ t_{0,n},t_{1,n},\ldots,t_{N_{n},n}\right\}$. The
COGARCH(p,q) state process $Y_{t}$ in \eqref{eq:StateProcYatTau} can
be defined equivalently $\forall\ n\geq n^{\star}$ as:
\begin{equation}
Y_{t_{i,n}}=C_{t_{i,n}}Y_{t_{i-1,n}}+D_{t_{i,n}}\label{eq:AlternativeCOGARHCPQ}
\end{equation}
with coefficients $C_{t_{i,n}}$ and $D_{t_{i,n}}$ defined as:
\begin{align*}
C_{t_{i,n}} & =\left(I+\Delta L_{t_{i,n}}^{2}\mathbf{ea}^{\top}\right)e^{B\Delta t_{i,n}}\\
D_{t_{i,n}} & =\alpha_{0}\Delta L_{t_{i,n}}^{2}\mathbf{e}.
\end{align*}
To show the ucp convergence of process $\tilde{Y}_{t,n}$ to $Y_{t}$,
we start observing that:
\begin{align}
\underset{t\in\left[0,T\right]}{\sup}\left\Vert Y_{t}-\tilde{Y}_{t,n}\right\Vert  & =\underset{t\in\left[0,T\right]}{\sup}\left\Vert e^{B\left(t-\Gamma_{t,n}\right)}\left(Y_{t_{i,n}}-\tilde{Y}_{i,n}\right)\right\Vert \nonumber \\
 & \leq e^{\left\Vert B\right\Vert _{M}T}\underset{i=1,\ldots,N_{n}}{\sup}\left\Vert \left(Y_{t_{i,n}}-\tilde{Y}_{i,n}\right)\right\Vert. \label{eq:LastInequalities}
\end{align}
We work on $\underset{i=1,\ldots,N_{n}}{\sup}\left\Vert \left(Y_{t_{i,n}}-\tilde{Y}_{i,n}\right)\right\Vert $
and for $i=1,\ldots,N_{n}$ and for fixed $n$ we have:
\begin{align}
\left\Vert \left(Y_{t_{i,n}}-\tilde{Y}_{i,n}\right)\right\Vert  & \leq\left\Vert \left(C_{t_{i,n}}Y_{t_{i-1,n}}-\tilde{C}_{i,n}\tilde{Y}_{i-1,n}\right)\right\Vert +\left\Vert D_{t_{i,n}}-\tilde{D}_{i,n}\right\Vert .\label{eq:LastIneq2}
\end{align}
The term $\left\Vert D_{t_{i,n}}-\tilde{D}_{i,n}\right\Vert $ in \eqref{eq:LastIneq2}
is bounded as follows:
\begin{equation}
\left\Vert D_{t_{i,n}}-\tilde{D}_{i,n}\right\Vert \leq\left|\alpha_{0}\right|\left|\Delta L_{t_{i,n}}^{2}-\mathbf{1}_{\tau_{i,n}<+\infty}\Delta L_{\tau_{i,n}}^{2}\right|.\label{eq:DiffD_t_i_mDtilde}
\end{equation}
Since
\begin{equation}
\Delta L_{t_{i,n}}^{2}:=\Delta L_{t_{i,n}}^{2}\mathbf{1}_{\left|\Delta L_{t_{i,n}}\right|\geq m_{n}}+\Delta L_{t_{i,n}}^{2}\mathbf{1}_{\left|\Delta L_{t_{i,n}}\right|<m_{n}},
\label{daric}
\end{equation}
the inequality in \eqref{eq:DiffD_t_i_mDtilde} becomes:
\begin{align}
\left\Vert D_{t_{i,n}}-\tilde{D}_{i,n}\right\Vert  & \leq\left|\alpha_{0}\right|\left|\Delta L_{t_{i,n}}^{2}\mathbf{1}_{0<\left|\Delta L_{t_{i,n}}\right|<m_{n}}\right|\nonumber \\
 & \leq m_{n}\left|\alpha_{0}\right|\left|\mathbf{1}_{\left|\Delta L_{t_{i,n}}\right|>0}\right|.\label{eq:ABC}
\end{align}
Inserting \eqref{eq:ABC} into \eqref{eq:LastIneq2}, we have:
\begin{align}
\left\Vert \left(Y_{t_{i,n}}-\tilde{Y}_{i,n}\right)\right\Vert  & \leq\left\Vert \left(C_{t_{i,n}}Y_{t_{i-1,n}}-\tilde{C}_{i,n}\tilde{Y}_{i-1,n}\right)\right\Vert +m_{n}\left|\alpha_{0}\right|\left|\mathbf{1}_{\left|\Delta L_{t_{i,n}}\right|>0}\right|\label{eq:LastIneq3}
\end{align}
We add and subtract the term $C_{t_{i,n}}\tilde{Y}_{i-1,n}$ into the
quantity $\left\Vert C_{t_{i,n}}Y_{t_{i-1,n}}-\tilde{C}_{i,n}\tilde{Y}_{i-1,n} \right\Vert $.
By exploiting the triangular inequality we obtain:
\footnotesize
\begin{align}
\left\Vert C_{t_{i,n}}Y_{t_{i-1,n}}-\tilde{C}_{i,n}\tilde{Y}_{i-1,n} \right\Vert  & \leq\left\Vert C_{t_{i,n}}Y_{t_{i-1,n}}-C_{t_{i,n}}\tilde{Y}_{i-1,n}\right\Vert +\left\Vert C_{t_{i,n}}\tilde{Y}_{i-1,n}-\tilde{C}_{i,n}\tilde{Y}_{i-1,n}\right\Vert \nonumber \\
 & \leq\left\Vert C_{t_{i,n}}\right\Vert _{M}\left\Vert Y_{t_{i-1,n}}-\tilde{Y}_{i-1,n}\right\Vert +\left\Vert C_{t_{i,n}}-\tilde{C}_{i,n}\right\Vert _{M}\left\Vert \tilde{Y}_{i-1,n}\right\Vert \nonumber \\
 & \leq\left\Vert C_{t_{i,n}}\right\Vert _{M}\left\Vert Y_{t_{i-1,n}}-\tilde{Y}_{i-1,n}\right\Vert \nonumber \\
 & +\left|\Delta L_{t_{i,n}}^{2}-\mathbf{1}_{\tau_{i,n}<+\infty}\Delta L_{\tau_{i,n}}^{2}\right|\left\Vert \mathbf{ea}^{\top}\right\Vert _{M}e^{\left\Vert B\right\Vert _{M}\Delta t_{i,n}}\left\Vert \tilde{Y}_{i-1,n}\right\Vert .\label{eq:CDE}
\end{align}
\normalsize
Defining:
\[
C_{t_{i,n}}^{\star\star}:=\left(1+\Delta L_{t_{i,n}}^{2}\left\Vert \mathbf{ea}^{\top}\right\Vert _{M}\right)e^{\left\Vert B\right\Vert _{M}\Delta t_{i,n}}\geq\left\Vert C_{t_{i,n}}\right\Vert _{M},
\]
substituting \eqref{eq:CDE} into \eqref{eq:LastIneq3} and using the same arguments as in \eqref{daric} and \eqref{eq:ABC}, we obtain:
\begin{align*}
\left\Vert Y_{t_{i,n}}-\tilde{Y}_{i,n}\right\Vert  & \leq C_{t_{i,n}}^{\star\star}\left\Vert Y_{t_{i-1,n}}-\tilde{Y}_{i-1,n}\right\Vert +m_{n}\left|\alpha_{0}\right|\left|\mathbf{1}_{\left|\Delta L_{t_{i,n}}\right|>0}\right|\\
 & +\left|\Delta L_{t_{i,n}}^{2}-\mathbf{1}_{\tau_{i,n}<+\infty}\Delta L_{\tau_{i,n}}^{2}\right|\left\Vert \mathbf{ea}^{\top}\right\Vert _{M}e^{\left\Vert B\right\Vert _{M}\Delta t_{i,n}}\left\Vert \tilde{Y}_{i-1,n}\right\Vert .
\end{align*}
Using $K_{i,n}$ in \eqref{eq:K_i_n}, we have:
\begin{align}
\left\Vert Y_{t_{i,n}}-\tilde{Y}_{i,n}\right\Vert  & \leq C_{t_{i,n}}^{\star\star}\left\Vert Y_{t_{i-1,n}}-\tilde{Y}_{i-1,n}\right\Vert +m_{n}K_{i-1,n}\left|\mathbf{1}_{\left|\Delta L_{t_{i,n}}\right|>0}\right|.\label{eq:LastIneq4}
\end{align}
We introduce a stochastic recurrence equation on the grid $\left\{ t_{i,n}\right\} _{i=0,\ldots,N_{n}}$ defined as
as:
\[
\zeta_{i,n}=C_{t_{i,n}}^{\star\star}\zeta_{i-1,n}+m_{n}K_{i-1,n}\mathbf{1}_{\left|\Delta L_{t_{i,n}}\right|>0}
\]
with initial condition $\zeta_{0,n}:=\left\Vert Y_{t_{0,n}}-\tilde{Y}_{0,n}\right\Vert =0$
a.s.. Since $\forall i\ C_{i,n}^{\star\star}\geq1$ and $m_{n}K_{i-1,n}^{\star}\mathbf{1}_{\left|\Delta L_{t_{i,n}}\right|>0}\geq0$
a.s., $\zeta_{i,n}$ is a non decreasing process that is an upper
bound for $\left\Vert Y_{t_{i,n}}-\tilde{Y}_{i,n}\right\Vert $ for each fixed $i$ then:
\footnotesize
\begin{align}
\underset{i=1,\ldots,N_{n}}{\sup}\left\Vert Y_{t_{i,n}}-\tilde{Y}_{i,n}\right\Vert  & \leq\left[\prod_{i=0}^{N_{n}-1}C_{N_{n}-i,n}^{\star\star}\right]\left\Vert Y_{t_{0,n}}-\tilde{Y}_{0,n}\right\Vert \nonumber \\
 & +m_{n}\left\{ \sum_{i=1}^{N_{n}-1}\left[\prod_{h=1}^{i}C_{N_{n}+i-h,n}^{\star\star}\right]\mathbf{1}_{\left|\Delta L_{t_{N_{n}-i,n}}\right|>0}K_{N_{n}-1-i,n}+\mathbf{1}_{\left|\Delta L_{t_{N_{n},n}}\right|>0}K_{N_{n}-1,n}\right\} \label{eq:LastIneq5}
\end{align}
\normalsize
The right-hand side in \eqref{eq:LastIneq5} is non-negative as a summation
of non-negative terms. We split it into two parts:
\footnotesize
\begin{align*}
G_{n} & :=\left[\prod_{i=0}^{N_{n}-1}C_{N_{n}-i,n}^{\star\star}\right]\left\Vert Y_{t_{0,n}}-\tilde{Y}_{0,n}\right\Vert \\
W_{n} & :=m_{n}\left\{ \sum_{i=1}^{N_{n}-1}\left[\prod_{h=1}^{i}C_{N_{n}+i-h,n}^{\star\star}\right]\mathbf{1}_{\left|\Delta L_{t_{N_{n}-i,n}}\right|>0}K_{N_{n}-1-i,n}+\mathbf{1}_{\left|\Delta L_{t_{N_{n},n}}\right|>0}K_{N_{n}-1,n}\right\}. 
\end{align*}
\normalsize
Using the same arguments as in \eqref{eq:PROVA}, we can say that:
\[
G_{n}=0\ \texttt{a.s.}\ \forall n\geq0
\]
and
\[
W_{n}\stackrel{n\rightarrow+\infty}{\rightarrow}0
\]
 since for $n\rightarrow+\infty$, the quantity
\[
\sum_{i=1}^{N_{n}-1}\left[\prod_{h=1}^{i}C_{N_{n}+i-h,n}^{\star\star}\right]\mathbf{1}_{\left|\Delta L_{t_{N_{n}-i,n}}\right|>0}K_{N_{n}-1-i,n}+\mathbf{1}_{\left|\Delta L_{t_{N_{n},n}}\right|>0}K_{N_{n}-1,n}
\]
is composed by a finite number of terms and then it finite a.s. for
the same arguments in \eqref{eq:Prova2}. In conclusion we have:
\[
\underset{i=1,\ldots,N_{n}}{\sup}\left\Vert Y_{t_{i,n}}-\tilde{Y}_{i,n}\right\Vert \leq G_{n}+W_{n}\underset{n\rightarrow+\infty}{\rightarrow}0
\]
that implies 
\begin{equation}
Y_{t,n}\stackrel{ucp}{\rightarrow}Y_{t}\label{eq:IMPORTANTE2}
\end{equation}
 where $Y_{t,n}$ is the constant piecewise process associated to
the process $Y_{i,n}$ defined in \eqref{eq:IMPORTANTEYI_N}. From \eqref{eq:IMPORTANTE2}
we obtain the ucp convergence of process $V_{i,n}$ to the COGARCH(p,q)
variance process $V_{t}$. The remaing part of the proof follows the
same steps as in \cite{maller2008garch}
\end{proof} 
The result can be generalized to any COGARCH(p,q) model driven by a finite variation L\'evy process since, as shown in \cite{brockwell2006}, a COGARCH(p,q) driven by a general L\'evy can by approximated by the same COGARCH(p,q) process driven by a Compound Poisson. Then using the triangular inequality, the discrete process $\left(G_{i,n},V_{i,n}\right)$ converges in the Skorokhod metric and in probability to any COGARCH(p,q) model.

\section{Maximum Pseudo-Loglikelihood Estimation for the COGARCH(p,q) process}

In this Section we show how to extend the maximum pseudo-loglikelihood
estimation procedure developed in \cite{maller2008garch} for
the COGARCH(1,1) model to the higher order case. We use the approximation scheme proposed in Section \ref{sec:2}
 as a generalization of the approach in \cite{maller2008garch} and used 
recently also in \cite{behme2014} for the GRJ-COGARCH(1,1) model. First of all, we recall the variance of the integrated COGARCH(p,q)
model on the interval $\left[t_{i-1},t_{i}\right]$ \citep[see][ for derivation of higher order moment of a COGARCH(p,q) process]{Chadraa2010Thesis}.
\newline
On the irregular grid 
\begin{equation}
0=t_{0}<t_{1}<\ldots<t_{N}=T\label{eq:IrregGridForPseudo}
\end{equation}
 we consider the increment of a COGARCH(p,q) process defined as:
\[
\Delta G_{t_{i}}:=G_{t_{i}}-G_{t_{i-1}}=\int_{t_{i-1}}^{t_{i}}V_{u}\mbox{d}L_{u}
\]
As shown in \cite{Chadraa2010Thesis}, the conditional first moment and the conditional variance are
respectively:

\begin{equation}
\begin{array}{l}
E\left[\Delta G_{t_{i}}\left|\mathcal{F}_{t_{i-1}}\right.\right]=0\\
Var\left[\Delta G_{t_{i}}\left|\mathcal{F}_{t_{i-1}}\right.\right]=E\left[L_1\right]\left[\frac{\alpha_{0}\Delta t_{i}b_{q}}{b_{q}-a_{1}\mu}+\mathbf{a}^{\top}e^{\tilde{B}\Delta t_{i}}\tilde{B}^{-1}\left(I-e^{-\tilde{B}\Delta t_{i}}\right)\left(Y_{t_{i-1}}-E\left(Y_{t_{i-1}}\right)\right)\right]
\end{array}\label{eq:VarianceCondForPseudo}
\end{equation}
where $\tilde{B}:=B+\mu\mathbf{ea}^{\top}$, $\mu=\int_{\mathcal{R}}y^{2}\mbox{d}\nu_{L}\left(y\right)$
and $\nu_{L}\left(y\right)$ is the L\'evy measure of the process $L_{t}$
for simplicity we require the underlying process to be centered in zero
with unitary second moment $\mu=E\left(L_{1}\right)=1$. Under
the assumption that guarantees the existence of the stationary mean of process $Y_{t}$ 
\citep[see][]{brockwell2006} we have:
\[
E\left(Y_{t}\right)=\frac{\alpha_{0}\mu}{b_{q}-a_{1}\mu}\left[\begin{array}{l}
1\\
0\\
\vdots\\
0
\end{array}\right].
\]
On the discrete grid in \eqref{eq:IrregGridForPseudo} we construct
the discrete process $G_{i,n}$ introduced in \eqref{eq:G_i_n}, in
particular we rewrite the state process $Y_{i,n}$ in \eqref{eq:IMPORTANTEYI_N} as follows:
\begin{align}
Y_{i,n} & =\left(I+\Delta t_{i,n}\epsilon_{i,n}^{2}\mathbf{ea}^{\top}\right)e^{B\Delta t_{i,n}}Y_{i-1,n}+\alpha_{0}\Delta t_{i,n}\epsilon_{i,n}^{2}\mathbf{e}\nonumber \\
 & =\left(I+\frac{\left(G_{i,n}-G_{i-1,n}\right)^{2}}{V_{i-1,n}}\mathbf{ea}^{\top}\right)e^{B\Delta t_{i,n}}Y_{i-1,n}+\alpha_{0}\frac{\left(G_{i,n}-G_{i-1,n}\right)^{2}}{V_{i-1,n}}\mathbf{e}\nonumber \\
 & =\left(I+\frac{\left(G_{i,n}-G_{i-1,n}\right)^{2}}{\alpha_{0}+\mathbf{a}^{\top}Y_{i-1,n}}\mathbf{ea}^{\top}\right)e^{B\Delta t_{i,n}}Y_{i-1,n}+\alpha_{0}\frac{\left(G_{i,n}-G_{i-1,n}\right)^{2}}{\alpha_{0}+\mathbf{a}^{\top}Y_{i-1,n}}\mathbf{e}.\label{eq:DynamicYinForPseudoLogLik}
\end{align}
Using the results \eqref{eq:VarianceCondForPseudo} and \eqref{eq:DynamicYinForPseudoLogLik},
we are able to generalize the pseudo-likelihood estimation procedure
in \cite{maller2008garch} for the case of the COGARCH(p,q) model. The
idea behind the pseudo-loglikelihood is based on the markovian property
of the pair $\left(G_{t},V_{t}\right)$ and the substitution of the
real transition density with the normality assumption with mean and
variance determined as in \eqref{eq:VarianceCondForPseudo}. Therefore
the maximum pseudo-loglikelihood estimates are obtained as solution
of the following optimization problem:
\[
\begin{array}{l}
\underset{\mathbf{a},\alpha_{0},B\in\Theta}{\max}\mathcal{L}_{N}\left(\mathbf{a},\alpha_{0},B\right)\\
\texttt{s.t.}\\
\left\{ \begin{array}{l}
Y_{i,n}=\left(I+\frac{\left(G_{i,n}-G_{i-1,n}\right)^{2}}{\alpha_{0}+\mathbf{a}^{\top}Y_{i-1,n}}\mathbf{ea}^{\top}\right)e^{B\Delta t_{i,n}}Y_{i-1,n}+\alpha_{0}\frac{\left(G_{i,n}-G_{i-1,n}\right)^{2}}{\alpha_{0}+\mathbf{a}^{\top}Y_{i-1,n}}\mathbf{e}\\
i=0,1,\ldots,N
\end{array}\right.
\end{array}
\]
where
\[
\mathcal{L}_{N}\left(\mathbf{a},\alpha_{0},B\right)=-\frac{1}{2}\sum_{i=1}^{N}\left(\frac{\left(\Delta G_{t_{i}}\right)^{2}}{Var\left[\Delta G_{t_{i}}\left|\mathcal{F}_{t_{i-1}}\right.\right]}+\ln\left(Var\left[\Delta G_{t_{i}}\left|\mathcal{F}_{t_{i-1}}\right.\right]\right)\right)-\frac{N\ln\left(2\pi\right)}{2}
\]
 and the set $\Theta$ contains the model parameters that ensure
the stationarity, the existence of the mean of the state process $Y_{t}$ and the non-negativity of process $V_t$.

%
%
%
%
%
%
%
%


\end{document}